\documentclass[12pt, reqno]{amsart}
\usepackage{amssymb,amsthm,amsmath}
\usepackage[numbers,sort&compress]{natbib}
\usepackage{amssymb,amsmath}
\usepackage{amsfonts,dsfont}
\usepackage{mathrsfs}
\usepackage{latexsym}
\usepackage{amssymb}
\usepackage{amsthm}
\usepackage{indentfirst}
\usepackage{hyperref}
\usepackage{color}
\usepackage{extpfeil}
\date{\today}

\let\oldsection\section
\renewcommand\section{\setcounter{equation}{0}\oldsection}

\newtheorem{corollary}{Corollary}[section]
\newtheorem{theorem}{Theorem}[section]
\newtheorem{lemma}{Lemma}[section]
\newtheorem{proposition}{Proposition}[section]
\newtheorem{definition}{Definition}[section]

\allowdisplaybreaks
\begin{document}

\title[Uniqueness of Solutions to Viscous Primitive Equations]{Uniqueness of Leray-Hopf weak solutions to the two-dimensional viscous Primitive Equations}
\author{Jinkai~Li}
\address[Jinkai~Li]{South China Research Center for Applied Mathematics and Interdisciplinary Studies, School of Mathematical Sciences,
South China Normal University, Guangzhou 510631, China}
\email{jklimath@gmail.com; jklimath@m.scnu.edu.cn}
\author{Lintao~Liu}
\address[Lintao~Liu]{The Institute of Mathematical Sciences, The Chinese University of Hong Kong, Hong Kong}
\email{ltliu@math.cuhk.edu.hk}

\author{Edriss S.~Titi}
\address[Edriss S.~Titi]{Department of Mathematics, Texas A$\&$M University, College Station, TX 77843, USA.
Department of Applied Mathematics and Theoretical Physics, University of Cambridge,
Cambridge CB3 0WA, UK. Department of Computer Science and Applied Mathematics,
Weizmann Institute of Science, Rehovot 76100, Israel.}
\email{titi@math.tamu.edu; Edriss.Titi@maths.cam.ac.uk}

\author{Dong~Wang}
\address[Dong~Wang]{South China Research Center for Applied Mathematics and Interdisciplinary Studies, School of Mathematical Sciences,
South China Normal University, Guangzhou 510631, China}
\email{WDongter@m.scnu.edu.cn}
%
%
\begin{abstract}
Global existence of Leray-Hopf (type) weak solutions to the viscous primitive equations in two or three dimensions was established by Lions--Temam--Wang \cite{Lions1,Lions2,Lions3,Lions4} in the 1990s;
however, the corresponding uniqueness, no matter in two or three dimensions, has been a long standing open problem since then. In this paper,
we solve this open problem in the two-dimensional case. Precisely, we prove the uniqueness of Leray-Hopf weak solutions, with arbitrary $L^2$ initial data, to the viscous primitive equations in a periodic channel, subject to impermeability and stress-free boundary conditions on the solid boundaries.
\end{abstract}
%
\maketitle

\allowdisplaybreaks

\section{Introduction}
\label{sec1}
The primitive equations form a fundamental block in models of planetary oceanic and atmospheric dynamics;
see, e.g., Lewandowski \cite{Lewan}, Majda \cite{Majda}, Pedlosky \cite{Pedlosky}, Vallis \cite{Vallis}, Washington--Parkinson \cite{Washing}, and Zeng \cite{Zeng}.
In the context of the large-scale oceanic and atmospheric dynamics, an important feature is that the horizontal scale of the ocean
and atmosphere is much larger than the vertical ones. Due to this fact, the primitive equations are derived by taking the small aspect
ratio limit to the Navier-Stokes equations, which has been rigorous justified by Az\'{e}rad--Guill\'{e}n \cite{PA.FG} in the framework of weak solutions and by Li--Titi \cite{J.Li2,J.Li3} in the framework of strong solutions, see also \cite{FugigaHHK, FuGiHiHuKaW1,FuruGigaKas,GaoNeTang,LiTitiYuan,LiuTiti,Puzhou} for further developments.

Mathematical analysis on the primitive equations has been attracting mathematicians' attentions and many significant results have been established. Global existence of weak solutions was established by Lions--Temam--Wang \cite{Lions1,Lions2,Lions3,Lions4} in the 1990s. However, the uniqueness of weak solutions, in both two and three dimensions, has been a long standing open problem since then.
Note that this is different from the Navier-Stokes equations, as the uniqueness of weak solutions to the two dimensional incompressible Navier-Stokes equations has been known for long times, see, e.g., \cite{PccF,lady,RoTemma}.
Break through concerning the strong solutions was made by Cao-Titi \cite{Cao1}, where they established the global well-posedness of strong solutions to the three dimensional primitive equations, see also Kobelkov \cite{GMK} and Kukavica--Ziane \cite{IKMZ}. For further development toward this direction, one may refer to \cite{Boling,MHTKa,MHATK,Cao2,J.Li5,J.Li6,J.Li1,J.Li7, J.Li8,CaoLiTitiWang,Husalmart,LiYuan,fanghan}; in particular, global well-posedness of strong solutions to the primitive equations with only horizontal viscosity was established in \cite{J.Li1,J.Li7, J.Li8,CaoLiTitiWang,Husalmart,fanghan}.

As mentioned in the previous paragraph, uniqueness of weak solutions to the primitive equations has been an open problem for long time. Here, the weak solutions that we
mention are those that with initial data only in $L^2$, and thus not any further regularities beyond $L^2$ are assumed. In the case that some more regularities are assumed on
the initial data, the weak solutions will enjoy more regularities and, as a result, are unique. This was first observed by Bresch et al \cite{DFA}, where they proved the
uniqueness in two dimensional of the so called z-weak solutions, i.e. those solutions that with initial data $v_0, \partial_z v_0\in L^2$. Such uniqueness of z-weak solutions
was later extended in the three dimensional case, see Tachim-Medjo \cite{TTM}. Another kind of weak solutions that has been known to be unique is those that with initial data
in the space  of continuous functions, see Kukavica--Pei--Ziane \cite{IYWM}. One can even show that the unique still hold if making small $L^\infty$ perturbation to the z-weak solutions, see Li--Titi \cite{J.Li4}.

Note that in all the works \cite{DFA,TTM,IYWM,J.Li4} mentioned in the previous paragraph, additional regularities beyond $L^2$ integrability on the initial data are required to guarantee the uniqueness of weak solutions, and thus the unique of general weak solutions is still unsolved. In this work, we will address this problem in the two dimensional case.

Let $h$ be a given positive number and $\Omega=\mathbb T\times (-h,0)$ a periodic channel. Consider the primitive equations in $\Omega$
\begin{align*}
\begin{cases}
&\partial_t v+v\partial_x v+w\partial_z v-\Delta v+\partial_xp=0,\\
&\partial_zp=0,  \quad \partial_x v+\partial_z w=0,
\end{cases}
\end{align*}
where $p(x,t),v(x,z,t), w(x,z,t)$ are the pressure, horizontal velocity and vertical velocity, respectively.
Complement the above system with the following boundary and initial conditions
\begin{eqnarray}
  &v, w \,\text{are periodic in}\, x,\nonumber\\
  &\partial_z v|_{z=-h,0}=w|_{z=-h,0}=0,  \label{bcv}\\
  &v|_{t=0}=v_0,\nonumber
\end{eqnarray}
with a given initial datum $v_0\in L^2(\Omega)$. Note that \eqref{bcv} is a natural impermeability and stress-free condition on the solid boundaries $z=-h, 0$.

Thanks to \eqref{bcv}, one can rewrite the above system equivalently as
\begin{eqnarray}
&\partial_tv+v\partial_xv+w\partial_zv-\Delta v
+\partial_xp=0, \label{eq8}\\
&\partial_z p=\int_{-h}^0\partial_x vdz=0, \label{eq9}\\
&w(x,z,t)=-\int_{-h}^z \partial_x v(x,\xi,t)d\xi,\label{eq11}
\end{eqnarray}
subject to the following boundary and initial conditions
\begin{eqnarray}
  &v \text{ is periodic in}\, x\label{eq12}\\
  &\partial_z v|_{z=-h,0}=0,  \label{eq13}\\
  &v|_{t=0}=v_0. \label{eq14}
\end{eqnarray}
Note that the impermeability condition $w|_{z=-h,0}=0$ in \eqref{bcv} is automatically guaranteed by \eqref{eq9}--\eqref{eq11}.

Throughout this paper, we use $L^q(\Omega)$ and $W^{m,q}(\Omega)$ to denote the standard Lebesgue and Sobolev spaces, repectively. For $q=2$, we use $H^m$ to replace $W^{m,2}$. We always use $\|f\|_p$ to denote the $L^p(\Omega)$ norm of $f$.

The following space $\mathcal{H}$ will be used:
$$\mathcal{H}=\left\{v\in H^1(\Omega)\,\Bigg|\, v \text{ is periodic in } x, \text{ and }\partial_x\int_{-h}^0 vdz=0\right\}.$$
\begin{definition}\label{def1}
Given a function $v_0 \in L^2(\Omega)$. A function $v$ is called a global Leray-Hopf weak solution to system
(\ref{eq8})--(\ref{eq9}), subject to (\ref{eq12})--(\ref{eq14}), if the following hold:

(i) $v\in C([0,\infty);L_w^2(\Omega))\cap L_{loc}^2([0,\infty);\mathcal{H})$, where $L^2_w$ stands for the $L^2$ space equipped with the weak topology.

(ii) For any compactly supported in time function $\varphi \in C([0,\infty);C^1(\Omega)\cap \mathcal{H})\cap C^1([0,\infty);C(\Omega))$, the following equality holds:
\begin{align*}
\int_0^\infty \int_\Omega[-v\partial_t\varphi+ v\partial_xv \varphi+w\partial_zv \varphi+\nabla v:\nabla \varphi]\,dxdzdt\\
=\int_\Omega v_0(x,z)\varphi(x,z,0)dxdz,
\end{align*}
where the vertical velocity $w$ is given by \eqref{eq11}.

(iii) The following energy inequality holds:
\begin{align*}
\frac{1}{2}\|v\|_2^2(t)+\int_0^t\|\nabla v\|_2^2(s)ds\leq \frac{1}{2}\|v_0\|_2^2,
\end{align*}
for a.e. $t\in (0,\infty)$.
\end{definition}


Our main result is the following theorem.

\begin{theorem}
  \label{the1}

For any $v_0  \in L^2(\Omega)$, the weak solution to system (\ref{eq8})--(\ref{eq9}), subject to (\ref{eq11})--(\ref{eq14}), is unique.
\end{theorem}

This will be proven in the next section.

\section{the proof for Theorem \ref{the1}}

We start with the following lemma.

\begin{lemma}\label{lmsum1}
Let $\alpha=\{\alpha_k\}_{k\geq1}, \beta=\{\beta_k\}_{k\geq1}, \gamma=\{\gamma_k\}_{k\geq1}$. Then, for any integer $M\geq 1$, it holds that
\begin{align*}
&\left|\sum_{\ell,m\geq 1}\sqrt{\frac{\max\{\ell,m\}}{\min\{\ell,m\}}}\left( {\bf 1}_{\ell\neq m}\gamma_{|\ell-m|}+\gamma_{\ell+m}\right)\alpha_m \beta_\ell\right| \\
\leq& C\|\{\sqrt k\gamma_k\}_{k\geq1}\|_{\ell^2}\Bigg(\sqrt{\log(e+M)}\|\alpha\|_{\ell^2}\|\beta\|_{\ell^2}\\
&
+\frac{\|\{\sqrt{k}\alpha_k\}_{k\geq1}\|_{\ell^2}\|\beta\|_{\ell^2}+\|\alpha\|_{\ell^2}\|\{\sqrt{k}\beta_k\}_{k\geq1}\|_{\ell^2}}{\sqrt{M}}\Bigg)
\end{align*}
for an absolute positive constant $C$.
\end{lemma}

\begin{proof}
Without loss of generality, we may assume that all components of $\alpha, \beta, \gamma$ are nonnegative, otherwise we may consider the inequality for $\tilde\alpha, \tilde\beta, \tilde\gamma$ whose components are given by the absolute values of those of $\alpha, \beta, \gamma$. For any $M\geq 1$, one deduces that
\begin{align}
\sum_{k\geq 1}\frac{\alpha_k}{\sqrt{k}}=&\sum_{1\leq k\leq M}\frac{\alpha_k}{\sqrt{k}}+\sum_{k\geq M+1}\frac{\sqrt{k}\alpha_k}{k}\nonumber\\
\leq& \left(\sum_{1\leq k\leq M}\alpha_k^2\right)^{1/2}\left(\sum_{1\leq k\leq M}\frac{1}{k}\right)^{1/2}+\left(\sum_{k\geq M+1}k\alpha_k^2\right)^{1/2}\left(\sum_{k\geq M+1}\frac{1}{k^2}\right)^{1/2}\nonumber\\
\leq& \sqrt{2 \log (e+M)}\|\alpha\|_{\ell^2}+\frac{\|\{\sqrt{k}\alpha_k\}_{k\geq1}\|_{\ell^2}}{\sqrt M},\label{eqak-1}
\end{align}
where in the last step the facts that $\sum_{1\leq k\leq M}\frac{1}{k}\leq1+\int_1^M\frac{dx}{x} \leq 2\log(e+M)$
and $\sum_{k\geq M+1}\frac{1}{k^2}\leq\int_M^\infty\frac{dx}{x^2}=\frac1M$ have been used.
Similarly,
\begin{align}
\sum_{k\geq 1}\frac{\beta_k}{\sqrt{k}}\leq \sqrt{2 \log (e+M)}\|\beta\|_{\ell^2}+\frac{\|\{\sqrt{k}\beta_k\}_{k\geq1}\|_{\ell^2}}{\sqrt M}.\label{eqbk-1}
\end{align}
Split the sum $\sum_{\ell,m\geq 1}$ into three cases: $\ell\geq 2m, \ell\leq \frac m2,$ and $\frac m2<\ell<2m$, that is
\begin{align*}
  S:=&\sum_{\ell,m\geq 1}\sqrt{\frac{\max\{\ell,m\}}{\min\{\ell,m\}}}\left( {\bf 1}_{\ell\neq m}\gamma_{|\ell-m|}+\gamma_{\ell+m}\right)\alpha_m \beta_\ell\\
  =&\left(\sum_{\substack{\ell\geq 2m\\ m\geq1}}+\sum_{\substack{m\geq 2\ell\\ \ell\geq 1}}+\sum_{\substack{ \frac{m}{2}<\ell<2m\\ \ell,m\geq1}}\right)\sqrt{\frac{\max\{\ell,m\}}{\min\{\ell,m\}}}\left( {\bf 1}_{\ell\neq m}\gamma_{|\ell-m|}+\gamma_{\ell+m}\right)\alpha_m \beta_\ell\\
  =:&S_1+S_2+S_3.
\end{align*}
In case $\ell\geq 2m$, it follows $\ell\leq 2(\ell-m)$, and then
\begin{align}
S_1
=&\sum_{m\geq 1}\sum_{\ell\geq 2m}\sqrt{\frac{\ell}{m}}(\gamma_{\ell-m}+\gamma_{\ell+m})\alpha_m\beta_\ell\\
\leq& \sum_{m\geq 1}\frac{\alpha_m}{\sqrt{m}}\sum_{\ell\geq 2m}\left(\sqrt{2(\ell-m)}\gamma_{\ell-m}+\sqrt{\ell+m}\gamma_{\ell+m}\right)\beta_\ell\nonumber\\
\leq& \sum_{m\geq 1}\frac{\alpha_m}{\sqrt{m}}\left[\left(\sum_{\ell\geq 2m}2(\ell-m)\gamma_{\ell-m}^2\right)^{1/2}+
\left(\sum_{\ell\geq 2m}(\ell+m)\gamma_{\ell+m}^2\right)^{1/2}\right]\left(\sum_{\ell\geq 2m}\beta_\ell^2\right)^{1/2}\nonumber\\
\leq& C\|\{\sqrt{k}\gamma_k\}_{k\geq1}\|_{\ell^2}\|\beta\|_{\ell^2}\sum_{k\geq 1}\frac{\alpha_k}{\sqrt{k}}.\label{eqlgm-1}
\end{align}
In case $\ell\leq \frac{m}{2}$ or equivalently $m\geq 2\ell$, it follows $m\leq 2(m-\ell)$, and then
\begin{align}
S_2
=&\sum_{\ell\geq 1}\sum_{m\geq 2\ell}\sqrt{\frac{m}{\ell}}\left(\gamma_{m-\ell}+\gamma_{\ell+m}\right)\alpha_m\beta_\ell
\nonumber\\
\leq& \sum_{\ell\geq 1}\frac{\beta_\ell}{\sqrt{\ell}}\sum_{m\geq 2\ell}\left(\sqrt{2(m-\ell)}\gamma_{m-\ell}+\sqrt{\ell+m}\gamma_{\ell+m}\right)\alpha_m
\nonumber\\
\leq&\sum_{\ell\geq 1}\frac{\beta_\ell}{\sqrt{\ell}}
\left[\left(\sum_{\ell\geq 2m}2(\ell-m)\gamma_{\ell-m}^2\right)^{1/2}+
\left(\sum_{\ell\geq 2m}(\ell+m)\gamma_{\ell+m}^2\right)^{1/2}\right]\left(\sum_{m\geq 2\ell}\alpha_m^2\right)^{1/2}\nonumber\\
\leq&C\|\{\sqrt{k}\gamma_k\}_{k\geq1}\|_{\ell^2}\|\alpha\|_{\ell^2}\sum_{k\geq 1}\frac{\beta_k}{\sqrt{k}}.\label{eqmgl-2}
\end{align}
It remains to deal with the case $\frac m2<\ell<2m$, that is the estimate for $S_3$.
In this case, it follows $\sqrt{\frac{\max\{\ell,m\}}{\min\{\ell,m\}}}\leq\sqrt{2}$, and then
\begin{align}
S_3=&\sum_{\substack{ \frac{m}{2}<\ell<2m\\ \ell,m\geq1}}\sqrt{\frac{\max\{\ell,m\}}{\min\{\ell,m\}}}\left({\bf 1}_{\ell\neq m}\gamma_{|\ell-m|}+\gamma_{\ell+m}\right)\alpha_m\beta_\ell\nonumber\\
\leq &\sqrt{2}\left(\sum_{\substack{\frac{m}{2}<\ell<m\\ \ell,m\geq1}}\gamma_{m-l}\alpha_m\beta_\ell
+\sum_{\substack{ m<\ell<2m\\ \ell,m\geq1}}\gamma_{l-m}\alpha_m\beta_\ell+\sum_{\substack{ \frac{m}{2}<\ell<2m\\ \ell,m\geq1}}\gamma_{\ell+m}\alpha_m\beta_\ell\right)\nonumber\\
\overset{\triangle}{=} & \sqrt{2}\Big(I_1(\alpha,\beta,\gamma)+I_2(\alpha,\beta,\gamma)+I_3(\alpha,\beta,\gamma)\Big).\label{eqlm-3}
\end{align}
Estimates on $I_1$, $I_2$, and $I_3$ are given as follows. Note that
\begin{align}
I_2(\alpha,\beta,\gamma)=&\sum_{\substack{ \ell<m<2\ell\\ \ell,m\geq1}}\gamma_{m-\ell}\alpha_\ell\beta_m \quad(\text{by exchanging } \ell \text{ and } m)\nonumber\\
=&\sum_{\substack{ \frac{m}{2}<\ell<m\\ \ell,m\geq1}}\gamma_{m-\ell}\alpha_\ell\beta_m=I_1(\beta, \alpha, \gamma).\label{eqI1-1}
\end{align}
For $I_1$, one deduces that
\begin{align}
I_1(\alpha, \beta,\gamma)=&\sum_{\substack{\frac{m}{2}<\ell<m\\ \ell,m\geq1}}\gamma_{m-l}\alpha_m\beta_\ell
\xlongequal{n=m-\ell}\sum_{\substack{\frac{n+\ell}{2}<\ell<n+\ell \\ n+\ell\geq 1, \ell\geq 1}}\gamma_n\alpha_{n+\ell}\beta_\ell\nonumber\\
=&\sum_{\substack{\ell, n\geq 1\\ \ell>n}}\gamma_n\alpha_{n+\ell}\beta_\ell=\sum_{n=1}^\infty\sum_{\ell>n}\alpha_{n+\ell}
\beta_\ell\gamma_n\nonumber\\
=&\sum_{n=1}^M\sum_{\ell>n}\alpha_{n+\ell}\beta_\ell\gamma_n+
\sum_{n>M}\sum_{\ell>n}\sqrt{n+\ell}\alpha_{n+\ell}\beta_\ell\frac{\gamma_n}{\sqrt{\ell+n}}\nonumber\\
\leq &\sum_{n=1}^M\left(\sum_{\ell>n}\alpha_{n+\ell}^2\right)^\frac{1}{2}
\left(\sum_{\ell>n}\beta_\ell^2\right)^\frac{1}{2}\gamma_n\nonumber\\
&+\sum_{n>M}\left(\sum_{\ell>n}(n+\ell)\alpha_{n+\ell}^2\right)^\frac{1}{2}
\left(\sum_{\ell>n}\beta_\ell^2\right)^\frac{1}{2}\frac{\gamma_n}{\sqrt{n}}\nonumber\\
\leq& \|\alpha\|_{\ell^2}\|\beta\|_{\ell^2}\sum_{n=1}^M\gamma_n+\|\beta\|_{\ell^2}\|\{\sqrt{k}\alpha_k\}_{k\geq 1}\|_{\ell^2}\sum_{n>M}\frac{\gamma_n}{\sqrt{n}}.\label{eqI1-2}
\end{align}
Note that
\begin{align*}
\sum_{n=1}^M\gamma_n=&\sum_{n=1}^M\frac{\sqrt{n}\gamma_n}{\sqrt{n}}\leq \left(\sum_{n=1}^Mn\gamma_n^2\right)^\frac{1}{2}\left(\sum_{n=1}^M\frac{1}{n}\right)^\frac{1}{2}\\
\leq& \sqrt{2\log(e+M)}\|\{\sqrt{k}\gamma_k\}_{k\geq 1}\|_{\ell^2},
\end{align*}
and
\begin{align*}
\sum_{n>M}\frac{\gamma_n}{\sqrt{n}}=&\sum_{n>M}\frac{\sqrt{n}\gamma_n}{n}\leq\left(
\sum_{n>M}\frac{1}{n^2}\right)^\frac{1}{2}\left(\sum_{n>M}n\gamma_n^2\right)^\frac{1}{2}\\
\leq& M^{-\frac{1}{2}}\|\{\sqrt{k}\gamma_k\}_{k\geq 1}\|_{\ell^2},
\end{align*}
where the facts that $\sum_{n=1}^M\frac{1}{n}\leq 2\log(e+M)$ and $\sum_{n>M}\frac{1}{n^2}\leq \frac{1}{M}$ have been used again.
Therefore, one obtains from \eqref{eqI1-2} that
\begin{align*}
&I_1(\alpha,\beta,\gamma)\\
\leq& \|\{\sqrt{k}\gamma_k\}_{k\geq 1}\|_{\ell^2}\left(\sqrt{2\log(e+M)}\|\alpha\|_{\ell^2}\|\beta\|_{\ell^2}
+M^{-\frac{1}{2}}\|\beta\|_{\ell^2}\|\{\sqrt{k}\alpha_k\}_{k\geq 1}\|_{\ell^2}\right).
\end{align*}
Thanks to this and recalling \eqref{eqI1-2}, one gets that
\begin{align*}
&I_2(\alpha,\beta,\gamma)=I_1(\beta,\alpha,\gamma)\\
\leq &\|\{\sqrt{k}\gamma_k\}_{k\geq 1}\|_{\ell^2}\left(\sqrt{2\log(e+M)}\|\alpha\|_{\ell^2}\|\beta\|_{\ell^2}
+M^{-\frac{1}{2}}\|\alpha\|_{\ell^2}\|\{\sqrt{k}\beta_k\}_{k\geq 1}\|_{\ell^2}\right).
\end{align*}
Now we estimate $I_3(\alpha,\beta,\gamma)$. One deduces that
\begin{align*}
&I_3(\alpha,\beta,\gamma)=\sum_{\substack{ \frac{m}{2}<\ell<2m\\ \ell,m\geq1}}\gamma_{n+\ell}\alpha_m\beta_\ell\\
=&\sum_{m\geq 1}\sum_{\frac{m}{2}<\ell<2m}\sqrt{\ell+m}\gamma_{\ell+m}\beta_\ell\frac{\alpha_m}{\sqrt{\ell+m}}\\
\leq&\sum_{m\geq 1}\frac{\alpha_m}{\sqrt{m}}\left(\sum_{\frac{m}{2}<\ell<2m}(\ell+m)\gamma_{\ell+m}^2\right)^\frac{1}{2}
\left(\sum_{\frac{m}{2}<\ell<2m}\beta_\ell^2\right)^\frac{1}{2}\\
\leq&\|\{\sqrt{k}\gamma_k\}_{k\geq 1}\|_{\ell^2}\|\beta\|_{\ell^2}\sum_{m\geq 1}\frac{\alpha_m}{\sqrt{m}}.
\end{align*}
Substituting the estimates $I_1$, $I_2$, and $I_3$ into \eqref{eqlm-3}, it follows that
\begin{align}
S_3\leq&\sqrt{2}\left(I_1(\alpha,\beta,\gamma)+I_2(\alpha,\beta,\gamma)+I_3(\alpha,\beta,\gamma)\right)\nonumber\\
\leq& C\|\{\sqrt k\gamma_k\}_{k\geq1}\|_{\ell^2}\Bigg(\|\beta\|_{\ell^2}\sum_{m\geq 1}\frac{\alpha_m}{\sqrt{m}}+\sqrt{\log(e+M)}\|\alpha\|_{\ell^2}\|\beta\|_{\ell^2}\nonumber\\
&
+M^{-\frac{1}{2}}\|\{\sqrt{k}\alpha_k\}_{k\geq1}\|_{\ell^2}\|\beta\|_{\ell^2}+M^{-\frac{1}{2}}\|\alpha\|_{\ell^2}\|\{\sqrt{k}\beta_k\}_{k\geq1}\|_{\ell^2}\Bigg)
.\label{eqI1-I3}
\end{align}
Combing \eqref{eqlgm-1}, \eqref{eqmgl-2}, and \eqref{eqI1-I3}, and using \eqref{eqak-1} and \eqref{eqbk-1}, the conclusion follows.
\end{proof}

Notice that the above lemma is true for every integer $M\geq 1$, then we have
\begin{corollary}
Let $\alpha=\{\alpha_k\}_{k\geq1}, \beta=\{\beta_k\}_{k\geq1}, \gamma=\{\gamma_k\}_{k\geq1}$. Then, it holds that
\begin{align*}
&\left|\sum_{\ell,m\geq 1}\sqrt{\frac{\max\{\ell,m\}}{\min\{\ell,m\}}}\left( {\bf 1}_{\ell\neq m}\gamma_{|\ell-m|}+\gamma_{\ell+m}\right)\alpha_m \beta_\ell\right| \\
\leq& C\min_{\substack{M\in \mathbb{N}\\ M\geq 1}}\|\{\sqrt k\gamma_k\}_{k\geq1}\|_{\ell^2}\Bigg(\sqrt{\log(e+M)}\|\alpha\|_{\ell^2}\|\beta\|_{\ell^2}\\
&
+\frac{\|\{\sqrt{k}\alpha_k\}_{k\geq1}\|_{\ell^2}\|\beta\|_{\ell^2}+\|\alpha\|_{\ell^2}\|\{\sqrt{k}\beta_k\}_{k\geq1}\|_{\ell^2}}{\sqrt{M}}\Bigg)
\end{align*}
for an absolute positive constant $C$.
\end{corollary}

Denote
\begin{align*}
\mathcal V=\left\{f\in H^2(\Omega)\Big| f \text{ is periodic in } x, \partial_z f|_{z={-h,0}}=0\right\}.
\end{align*}
Every $f\in \mathcal V$ can be expanded as a Fourier series in the $z$-variable:
\begin{align*}
f(x,z)=f_0(x)+\sum_{k=1}^\infty f_k(x)\cos \frac{k\pi z}{h},
\end{align*}
where
\begin{equation}
f_0(x)=\frac{1}{h}\int_{-h}^0f(x,z)dz,\quad f_k(x)=\frac{2}{h}\int_{-h}^0f(x,z)\cos \frac{k \pi z}{h}dz.\label{2.9}
\end{equation}
For $ s\leq 2$, define the operator $\Lambda_z^s$ on $\mathcal V$ as
\begin{equation}\label{2.10}
  \left(\Lambda_z^sf\right)(x,z)=f_0(x)+\sum_{k=1}^\infty k^sf_k(x)\cos\frac{k\pi z}{h},\quad \forall f\in \mathcal V,
\end{equation}
with $f_0(x)$ and $f_k(x)$ given as in \eqref{2.9}. Note that for $f_0(x)\equiv 0$, $\Lambda_z$ is invertible.

\begin{proposition}\label{prouni1}
Let $a,b\in \mathcal{V}$ satisfy $\int_{-h}^0adz=\partial_x\int_{-h}^0 bdz=0$, and denote
\begin{equation}
w_a(x,z)=-\int_{-h}^z\partial_xa(x,z')dz',
\quad
w_b(x,z)=-\int_{-h}^z\partial_xb(x,z')dz'.\label{wab}
\end{equation}
 Then, for any integral $M\geq 1$, it holds that
\begin{align*}
&\int_\Omega \left(ab_x+w_a\partial_zb+ba_x+w_b\partial_za\right)\Lambda_z^{-1}a\,dxdz\\
\leq &C\left(\sum_{k\geq 1}k\|b_k\|_{L^\infty_x}^2\right)^{\frac12}\left\{\sqrt{\log(e+M)}\|\Lambda_z^{-\frac12}a\|_{2}
\|\Lambda_z^{-\frac12}a_x\|_{2}\right.\\
&+M^{-\frac12}\|a\|_2\|\Lambda_z^{-\frac12}a_x\|_2+M^{-1/2}\|\Lambda_z^{-\frac12}a\|_2\|a_x\|_2\Big\},
\end{align*}
for a positive constant $C$ depending only on $h$, where $b_k$ and $\Lambda_z^{-\frac12}$ are defined as in \eqref{2.9} and \eqref{2.10}, respectively.
\end{proposition}
\begin{proof}
By direct computations, one can check easily that
\begin{align}
&\int_{-h}^0 \cos\frac{\ell\pi z}{h}\cos\frac{m\pi z}{h}\cos\frac{n\pi z}{h}dz=\frac{h}{4}\left({\bf 1}_{\ell=m+n}+{\bf 1}_{n=m+\ell}+{\bf 1}_{m=\ell+n}\right),\label{2.11}\\
&\int_{-h}^0 \cos\frac{\ell\pi z}{h}\sin\frac{m\pi z}{h}\sin\frac{n\pi z}{h}dz=\frac{h}{4}\left(-{\bf 1}_{\ell=m+n}+{\bf 1}_{n=m+\ell}+{\bf 1}_{m=\ell+n}\right).\label{2.12}
\end{align}
Recalling that $\int_{-h}^0adz=\partial_x\int_{-h}^0bdz=0$, one has

\begin{align*}
&a(x,z)=\sum_{k=1}^\infty a_k(x) \cos\frac{k\pi z}{h},\quad
b(x,z)=b_0(x)+\sum_{k=1}^\infty b_k(x) \cos\frac{k\pi z}{h},\quad \partial_xb_0(x)=0,
\end{align*}
and thus
\begin{align*}
\left(\Lambda_z^{-1}a\right)(x,z)=\sum_{k=1}^\infty \frac{a_k(x)}{k}\cos\frac{k\pi z}{h}.
\end{align*}
Then, direct computations yield
\begin{align*}
& w_a(x,z)=-\frac{h}{\pi}\sum_{k=1}^\infty \frac{\partial_xa_{k}(x)}{k}\sin\frac{k\pi z}{h},\quad
\partial_z a(x,z)=-\frac{\pi}{h}\sum_{k=1}^\infty ka_k(x)\sin\frac{k\pi z}{h},\\
& w_b(x,z)=-\frac{h}{\pi}\sum_{k=1}^\infty \frac{\partial_xb_{k}(x)}{k}\sin\frac{k\pi z}{h},\quad
\partial_z b(x,z)=-\frac{\pi}{h}\sum_{k=1}^\infty kb_k(x)\sin\frac{k \pi z}{h}.
\end{align*}
Consequently, one obtains by \eqref{2.11} that
\begin{align}
\int_\Omega ab_x\Lambda_z^{-1}a\,dxdz=&\sum_{m,n,\ell\geq 1}\int_\mathbb T \partial_xb_{n}a_\ell \frac{a_m}{m}dx\int_{-h}^0\cos\frac{\ell\pi z}{h}\cos\frac{m\pi z}{h}\cos\frac{n\pi z}{h}dz\nonumber\\
=&\sum_{m,n,\ell\geq1}\frac{h}{4m}\int_\mathbb T\partial_xb_{n}a_\ell a_mdx\left({\bf 1}_{\ell=m+n}+{\bf 1}_{n=m+\ell}+{\bf 1}_{m=\ell+n}\right),\label{2.13}
\end{align}
and similarly
\begin{equation}
\label{2.14}
\int_\Omega ba_x\Lambda_z^{-1}a\,dxdz=\sum_{m,n,\ell\geq 1}\frac{h}{4m}\int_\mathbb Tb_{n}\partial_xa_{\ell}a_mdx\left({\bf 1}_{\ell=m+n}+{\bf 1}_{n=m+\ell}+{\bf 1}_{m=\ell+n}\right),
\end{equation}
where $\partial_x b_0(x)=0$ is used to guarantee $\int_\Omega b_0a_x\Lambda_z^{-1}a dxdz=0$. In fact,
\begin{align*}
\int_\Omega b_0a_x\Lambda_z^{-1}a\,dxdz=& \sum_{m,\ell\geq 1}\frac{b_0}{m}\int_\mathbb T\partial_xa_{\ell}a_mdx\int_{-h}^0\cos\frac{m\pi z}{h}\cos\frac{\ell\pi z}{h}dz\\
=&\frac{hb_0}{2}\sum_{m\geq 1}\frac{1}{m}\int_\mathbb T\partial_xa_{m}a_mdx=0.
\end{align*}
Summing \eqref{2.13} with \eqref{2.14}, it follows from integration by part that
\begin{align}
\int_\Omega (ba_x+ab_x)\Lambda_z^{-1}a\,dxdz=-\sum_{m,n,\ell\geq 1}\frac{h}{4m}\int_\mathbb Tb_{n}a_{\ell}\partial_xa_{m}dx\left({\bf 1}_{\ell=m+n}+{\bf 1}_{n=m+\ell}+{\bf 1}_{m=\ell+n}\right)\label{equni1}.
\end{align}
Next, we have by \eqref{2.12} that
\begin{align}
&\int_\Omega w_a\partial_z b\Lambda_z^{-1} a\,dxdz\\
=&\sum_{m,n,\ell\geq 1}\frac{n}{\ell m} \int_\mathbb T\partial_xa_{\ell}b_na_mdx\int_{-h}^0\sin\frac{\ell\pi z}{h}\sin\frac{n\pi z}{h}\cos\frac{m\pi z}{h}dz\nonumber\\
=&\frac{h}{4}\sum_{m,n,\ell\geq 1}\frac{n}{\ell m} \int_\mathbb T\partial_xa_{\ell}b_na_mdx\left(-{\bf 1}_{m=\ell+n}+{\bf 1}_{n=m+\ell}+{\bf 1}_{\ell=m+n}\right)\nonumber\\
=&\frac{h}{4}\sum_{m,n,\ell\geq 1}\frac{n}{\ell m} \int_\mathbb Ta_{\ell}b_n\partial_xa_{m}dx\left(-{\bf 1}_{\ell=m+n}+{\bf 1}_{n=m+\ell}+{\bf 1}_{m=\ell+n}\right),\label{equni2}
\end{align}
 where, in the last step, we have exchanged the labels $\ell$ and $m$. Similarly
\begin{align}
&\int_\Omega w_b\partial_z a\Lambda_z^{-1} a\,dxdz\nonumber\\
=&
\sum_{m,n,\ell\geq 1}\frac{\ell}{nm}\int_\mathbb T\partial_xb_{n}a_\ell a_mdx\int_{-h}^0\sin\frac{n\pi z}{h}\sin\frac{\ell\pi z}{h}\cos\frac{m\pi z}{h}dz\nonumber\\
=&\frac{h}{4}\sum_{m,n,\ell\geq 1}\frac{\ell}{nm} \int_\mathbb T\partial_xb_{n}a_\ell a_mdx\left(-{\bf 1}_{m=\ell+n}+{\bf 1}_{n=m+\ell}+{\bf 1}_{\ell=m+n}\right)\nonumber\\
=&-\frac{h}{4}\sum_{m,n,\ell\geq 1}\frac{\ell}{nm} \int_\mathbb Tb_{n}\partial_xa_{\ell}a_mdx\left(-{\bf 1}_{m=\ell+n}+{\bf 1}_{n=m+\ell}+{\bf 1}_{\ell=m+n}\right)\nonumber\\
&-\frac{h}{4}\sum_{m,n,\ell\geq 1}\frac{\ell}{nm} \int_\mathbb Tb_{n}a_\ell\partial_xa_{m}dx\left(-{\bf 1}_{m=\ell+n}+{\bf 1}_{n=m+\ell}+{\bf 1}_{\ell=m+n}\right)\nonumber\\
=&-\frac{h}{4}\sum_{m,n,\ell\geq 1}\frac{m}{n\ell} \int_\mathbb Tb_{n}a_{\ell}\partial_xa_{m}dx\left(-{\bf 1}_{\ell=m+n}+{\bf 1}_{n=m+\ell}+{\bf 1}_{m=\ell+n}\right)\nonumber\\
&-\frac{h}{4}\sum_{m,n,\ell\geq 1}\frac{\ell}{nm} \int_\mathbb Tb_{n}a_\ell\partial_xa_{m}dx\left(-{\bf 1}_{m=\ell+n}+{\bf 1}_{n=m+\ell}+{\bf 1}_{\ell=m+n}\right)\label{equni3},
\end{align}
where, in the last step, we have exchanged the labels $\ell$ and $m$ for the first term.

Summing \eqref{equni1}, \eqref{equni2} and \eqref{equni3}, and noticing that the coefficients for ${\bf 1}_{m=\ell+n}$, ${\bf 1}_{n=m+\ell}$, and ${\bf 1}_{\ell=m+n}$ are $-\frac{3h}{4m}$, $\frac{(m-\ell)h}{4mn}$, and $-\frac{3h}{4m}$, respectively, it follows from Lemma \ref{lmsum1} that
\begin{align*}
&\int_\Omega \left(ab_x+w_a\partial_zb+ba_x+w_b\partial_za\right)\Lambda_z^{-1}adxdz\\
=&-\frac{3h}{4}\sum_{\substack{\ell,m,n\geq 1\\ \ell=m+n}}\frac{1}{m}\int_\mathbb Tb_na_\ell\partial_xa_{m}dx-\frac{3h}{4}\sum_{\substack{\ell,m,n\geq 1\\m=\ell+n}}\frac{1}{m}\int_\mathbb Tb_na_\ell\partial_xa_{m}dx\\
&+\frac{h}{4}\sum_{\substack{\ell,m,n\geq 1\\n=m+\ell}}\frac{m-\ell}{mn}\int_\mathbb Tb_na_\ell\partial_xa_{m}dx\\
=&-\frac{3h}{4}\sum_{\substack{\ell,m\geq 1\\ \ell\neq m}}\frac{1}{m}\int_\mathbb Tb_{|\ell-m|}a_\ell\partial_xa_{m}dx+\frac{h}{4}\sum_{\substack{\ell,m\geq 1}}\frac{m-\ell}{m(m+\ell)}\int_\mathbb Tb_{m+\ell}a_\ell\partial_xa_{m}dx\\
\leq& C\sum_{\substack{\ell,m\geq 1\\ \ell\neq m}}\frac{1}{m}\|b_{|\ell-m|}\|_{L^\infty _x}\|a_\ell\|_{L^2_x}\|\partial_xa_{m}\|_{L^2_x}+C\sum_{\substack{\ell,m\geq 1}}\frac{1}{m}\|b_{m+\ell}\|_{L^\infty_x}\|a_\ell\|_{L^2_x}\|\partial_xa_{m}\|_{L^2_x}\\
=&C\sum_{\substack{\ell,m\geq 1\\ \ell\neq m}}\sqrt{\frac{\ell}{m}}\left(\|b_{|\ell-m|}\|_{L^\infty _x}{\bf 1}_{\ell\neq m}+\|b_{m+\ell}\|_{L^\infty_x}\right)\left\|\frac{a_\ell}{\sqrt{\ell}}\right\|_{L^2_x}\left\|\frac{\partial_xa_{m}}{\sqrt{m}}\right\|_{L^2_x}\\
\leq& C\sum_{\ell,m\geq 1}\sqrt{\frac{\max\{\ell,m\}}{\min\{\ell,m\}}}\left({\bf 1}_{\ell\neq m}\|b_{|\ell-m|}\|_{L^\infty_x}+\|b_{m+\ell}\|_{L^\infty_x}\right)\left\|\frac{a_\ell}{\sqrt{\ell}}\right\|_{L^2_x}\left\|\frac{\partial_xa_{m}}{\sqrt{m}}\right\|_{L^2_x}\\
\leq& C\left(\sum_{k\geq 1}k\|b_k\|_{L^\infty_x}^2\right)^{\frac12}\left\{\sqrt{\log(e+M)}\|\Lambda_z^{-\frac12}a\|_{2}
\|\Lambda_z^{-\frac12}a_x\|_{2}\right.\\
&+M^{-\frac12}\|a\|_2\|\Lambda_z^{-\frac12}a_x\|_2+M^{-\frac12}\|\Lambda_z^{-\frac12}a\|_2\|a_x\|_2\Big\},
\end{align*}
proving the conclusion.
\end{proof}

Before giving the proof of the uniqueness, let us state the following proposition on the regularities of Leray-Hopf weak solutions to the primitive equations.

\begin{proposition}
\label{propreg}
  Let $v$ be a Leray-Hopf weak solution to system (\ref{eq8})--(\ref{eq9}), subject to (\ref{eq12})--(\ref{eq14}). Then, the following hold:

(i) $v$ is smooth away from the initial time, i.e.~$v\in C^\infty(\overline\Omega\times(0,\infty))$;

(ii) $v$ satisfies the energy equality
$$
\frac12\|v\|_2^2(t)+\int_0^t\|\nabla v\|_2^2(\tau)d\tau=\frac12\|v_0\|_2^2,\quad\forall t\in(0,\infty);
$$

(iii) $v\in C([0,\infty); L^2(\Omega))$.
\end{proposition}

\begin{proof}
Points (i) and (ii) are exactly those in Corollary 3.1 of Li--Titi \cite{J.Li4}. Concerning (iii), since (i) implies that $v\in C((0,\infty); L^2(\Omega))$, it remains to show that $v(t)\rightarrow v_0$ in $L^2(\Omega)$ as $t\rightarrow0$. This follows from the facts that $v(t) \rightharpoonup v_0$ in $L^2(\Omega)$ and that $\|v\|_2(t)\rightarrow\|v_0\|_2$ as $t\rightarrow0$, which are, respectively, guaranteed by point (i) of Definition \ref{def1}, and the energy equality in point (ii) of this proposition.
\end{proof}

We are now ready to give the proof of Theorem \ref{the1}.

\begin{proof}[\bf Proof of Theorem \ref{the1}. \rm]
 Let $v_1$ and $v_2$ be two Leray-Hopf weak solutions to
 system (\ref{eq8})--(\ref{eq9}), subject to (\ref{eq12})--(\ref{eq14}), with the same initial data $v_{0}$. Set $a=v_1-v_2$ and $b=\frac{v_1+v_2}{2}$. Then, $(a,b)$ satisfies
 \begin{eqnarray}
  &\displaystyle\partial_ta-\Delta a+ab_x+w_a\partial_zb+ba_x+w_b\partial_za+\partial_xp(x,t)=0,\label{eqcontv}\\
& \partial_x\int_{-h}^0{a}dz=\partial_x\int_{-h}^0{b}dz=0, \label{eqv0a0}
\end{eqnarray}
subject to the boundary conditions (\ref{eq12})--(\ref{eq13}), with initial data $a|_{t=0}=0$, where $w_a$ and $w_b$ are given as in \eqref{wab}. Due to Proposition \ref{propreg}, it is clear that $a, b\in C^\infty(\overline\Omega\times(0,\infty))$ and
\eqref{eqcontv} and \eqref{eqv0a0} are satisfied in classical sense in $\Omega\times(0,\infty)$.

Integrating \eqref{eqcontv} over $\Omega$ and using the boundary conditions (\ref{eq12})--(\ref{eq13}), one obtains that $\partial_t\int_\Omega{a}dxdz=0$ for any $t\in(0,\infty)$. Thanks to this and since $a\in C([0,T);L^2(\Omega))$ guaranteed by (iii) of Proposition \ref{propreg}, one has
$$
\left(\int_\Omega{a}dxdz\right)(t)=\left(\int_\Omega{a}dxdz\right)(0)=0,
$$
that is
$\int_\mathbb T\int_{-h}^0{a}dzdx=0$, for any $t\in(0,\infty)$. This and \eqref{eqv0a0} imply that $\int_{-h}^0{a}dz=0$.

Recall that $b_k(x,t)=\int_{-h}^0b(x,z,t)\cos\frac{k\pi z}{h}dz$.
It follows from the Gagliardo-Nirenberg inequality that
\begin{align}
&\sum_{k\geq 1}k\|b_k\|_{L^\infty_x}^2\leq C\sum_{k\geq 1}k\|b_k\|_{L^2_x}(\|b_k\|_{L^2_x}+\|\partial_xb_{k}\|_{L^2_x})\nonumber\\
\leq& C \left(\sum_{k\geq 1} k^2\|b_k\|_{L^2_x}^2\right)^{\frac12}\left[\left(\sum_{k\geq 1} \|b_k\|_{L^2_x}^2\right)^{\frac12}+\left(\sum_{k\geq 1} \|\partial_xb_{k}\|_{L^2_x}^2\right)^{\frac12}\right]\nonumber\\
=&C\left(\int_\mathbb T\sum_{k\geq 1}k^2b_k^2dx\right)^\frac{1}{2}\left[\left(\int_\mathbb T\sum_{k\geq 1}b_k^2dx\right)^\frac{1}{2}+\left(\int_\mathbb T\sum_{k\geq 1}|\partial_xb_k|^2dx\right)^\frac{1}{2}\right]\nonumber\\
=&C\|\partial_zb\|_2\left(\|b\|_2+\|\partial_xb\|_2\right)\leq C\|b\|_{H^1}^2.\label{eqnb-1}
\end{align}

Note that the function $a$ is smooth on $\overline\Omega\times(0,\infty)$, so is $\Lambda_z^{-1}a$.
Multiplying (\ref{eqcontv}) with $\Lambda_z^{-1}a$, integrating over $\Omega$, and recalling $\int_{-h}^0adz=\partial_x\int_{-h}^0bdz=0$,
it follows from Proposition \ref{prouni1} and \eqref{eqnb-1} that for every $t\in(0,\infty)$ one has
\begin{align}
&\frac{h}{4}\frac{d}{dt}\|\Lambda_z^{-\frac12}a\|_2^2+\frac{h}{2}\left(\| \Lambda_z^{\frac12}a\|_2^2+\|\Lambda_z^{-\frac12}a_x\|_2^2\right)\nonumber\\
=&\int_\Omega(ab_x+w_a\partial_zb+ba_x+w_b\partial_za)\Lambda_z^{-1}adxdz\nonumber\\
\leq& C\left(\sum_{k\geq 1}k\|b_k\|_{L^\infty_x}^2\right)^{\frac12}\left\{\sqrt{\log(e+M)}\|\Lambda_z^{-\frac12}a\|_{2}
\|\Lambda_z^{-\frac12}a_x\|_{2}\right.\nonumber\\
&+M^{-\frac12}\|a\|_2\|\Lambda_z^{-\frac12}a_x\|_2+M^{-\frac12}\|\Lambda_z^{-\frac12}a\|_2\|a_x\|_2\Big\}\nonumber\\
\leq&C\|b\|_{H^1}\left\{\sqrt{\log(e+M)}\|\Lambda_z^{-\frac12}a\|_{2}
\|\Lambda_z^{-\frac12}a_x\|_{2}\right.\nonumber\\
&+M^{-\frac12}\|a\|_2\|\Lambda_z^{-\frac12}a_x\|_2+M^{-\frac12}\|\Lambda_z^{-\frac12}a\|_2\|a_x\|_2\Big\},\label{equni-1}
\end{align}
where the following has been used
\begin{align*}
\int_\Omega\partial_x pa_m\cos\frac{m\pi z}{h}dxdz=\int_\mathbb T \partial_xp a_mdx\int_{-h}^0\cos\frac{m\pi z}{h}dz=0,\quad \forall m\geq 1.
\end{align*}

Set $E_\varepsilon=\|\Lambda_z^{-\frac12}a\|_2^2+\varepsilon$
for $0<\varepsilon <1$. Recalling the definition of $\Lambda_z^{-\frac12}$ and properties of weak solutions, one has
$\|\Lambda_z^{-\frac12}a\|_2^2\leq \|a\|_2^2=\|v_1-v_2\|_2^2\leq 2\|v_0\|_2^2,$
and thus
$0< E_\varepsilon <R$, for a positive constant $R$. Set $M=\lceil \frac{R}{E_\varepsilon} \rceil$. Then, it is clear that $M^{-\frac12}\leq \sqrt{\frac{E_\varepsilon}{R}}$ and $\log(e+M)\leq C\log\left(e+\frac{R}{E_\varepsilon}\right)$. Thanks to these and using \eqref{equni-1}, it follows from the Young inequality that
\begin{align}
\|b\|_{H^1}\sqrt{\log(e+M)}\|\Lambda_z^{-\frac12}a\|_{2}
\|\Lambda_z^{-\frac12}a_x\|_{2}
\leq & \|b\|_{H^1}\sqrt{\log(e+R/E_\varepsilon)}E_\varepsilon^{\frac12}\|\Lambda_z^{-\frac12}a_x\|_{2}\nonumber\\
\leq& \eta \|\Lambda_z^{-\frac12}a_x\|_{2}^2+C_\eta\|b\|_{H^1}^2E_\varepsilon\log(e+R/E_\varepsilon),\label{eqA}\\
\|b\|_{H^1}M^{-\frac12}\|a\|_2\|\Lambda_z^{-\frac12}a_x\|_2\leq& \|b\|_{H^1}\sqrt{\frac{E_\varepsilon}{R}}\|a\|_2\|\Lambda_z^{-\frac12}a_x\|_2\nonumber\\
\leq& \eta \|\Lambda_z^{-\frac12}a_x\|_2^2+C_\eta\|b\|_{H^1}^2E_\varepsilon,\label{eqB}
\end{align}
and
\begin{align}
\|b\|_{H^1}M^{-\frac12}\|\Lambda_z^{-\frac12}a\|_2\|a_x\|_2\leq & \|b\|_{H^1}\sqrt{\frac{E_\varepsilon}{R}}\sqrt{E_\varepsilon}\|a_x\|_2\leq C\|b\|_{H^1}\|a_x\|_{H^1}E_\varepsilon,\label{eqC}
\end{align}
for any $\eta>0$.
Set $f(t)=\|b\|_{H^1}^2+\|b\|_{H^1}\|a_x\|_2$. Then, by the regularities of $v_1$ and $v_2$, it is clear that $f(t)\in L^1((0,T))$. Substituting \eqref{eqA}--\eqref{eqC} into \eqref{equni-1}, choosing $0<\eta\ll 1$, and noticing that $e+R/E_\varepsilon\leq (e+1)R/E_\varepsilon\leq 2eR/E_\varepsilon$, one obtains
\begin{align*}
\frac{d}{dt}E_\varepsilon(t)\leq Cf(t)E_\varepsilon(t)\log(e+R/E_\varepsilon)\leq Cf(t)E_\varepsilon(t)\log(2eR/E_\varepsilon), \quad \text{ for } t>0,
\end{align*}
from which, one gets further that
\begin{align*}
E_\varepsilon(t)\leq \left(E_\varepsilon(\delta)\right)^{e^{-C\int_\delta^tf(s)ds}}(2eR)^{C\int_\delta^tf(s)ds}, \quad\forall \,\, 0<\delta<t<\infty.
\end{align*}
Recalling that $a\in C([0,T);L^2(\Omega))$, one has $\Lambda_z^{-\frac12}a\in C([0,T);L^2(\Omega))$, and thus $E_\varepsilon(t)\in C([0,T))$ and $E_\varepsilon(\delta)\rightarrow E_\varepsilon(0)=\varepsilon$ as $\delta\rightarrow 0^+$.
Thanks to this, by choosing first $\delta\rightarrow 0^+$ and then $\varepsilon\rightarrow 0^+$, one gets $\|\Lambda_z^{-\frac12}a\|_2^2=0$, and thus $a=0$. This concludes the proof of the uniqueness.
\end{proof}

\section*{Acknowledgments}

The work of JL was supported in part by the National Nature Science Foundation of
China (Grant No. 12371204 and No. 12526408) and Guangdong Basic and Applied Basic Research
Foundation (Grant No. 2026A1515010778). The work of EST was supported in part by the DFG Research Unit FOR 5528
on Geophysical Flows.

\end{document}